\documentclass{amsart}
\usepackage{fourier,utopia}
\usepackage{ulem}
\usepackage{amssymb}
\usepackage{amsthm}
\usepackage{mathrsfs}
\usepackage{eufrak}
\usepackage{graphicx}
\usepackage{ stmaryrd }
\usepackage{float}
\usepackage{amsmath}
\usepackage[all]{xy}\usepackage{xypic}
\usepackage{hyperref}
\hypersetup{%
  pdftitle   = {},
  pdfauthor  = {},
  pdfcreator = {\LaTeX\ with package \flqq hyperref\frqq}
}

\DeclareMathAlphabet{\mathpzc}{OT1}{pzc}{m}{it}

\newtheorem{theorem}{Theorem}[section]
\newtheorem*{theorem*}{Theorem}

\newtheorem{lemma}[theorem]{Lemma}
\newtheorem*{lemma*}{Lemma}
\newtheorem{corollary}[theorem]{Corollary}

\theoremstyle{definition}
\newtheorem{definition}[theorem]{Definition}

\newtheorem{notation}[theorem]{Notation}
\theoremstyle{remark}
\newtheorem{remark}[theorem]{Remark}

\newtheorem{question}[theorem]{Question}

\DeclareMathOperator{\Ext}{Ext}

\DeclareMathOperator{\Pic}{Pic}

\DeclareMathOperator{\Tot}{Tot}

\numberwithin{equation}{section}

\begin{document}

\title{Isomorphisms of Moduli Spaces}

\author{C. Casorr\'an Amilburu}

\author{S. Barmeier}

\author{B. Callander}

\author{E. Gasparim}
\address{Elizabeth Gasparim and Brian Callander}
\address{ IMECC -- UNICAMP,
Rua S\'ergio Buarque de Holanda, 651,
 Cidade Universit\'aria Zeferino Vaz,
Distr. Bar\~ao Geraldo, Campinas SP,  Brasil
13083-859 }
\email{gasparim@ime.unicamp.br}
\email{briancallander@googlemail.com}

\address{Carlos Casorr\'an Amilburu,
Depto. de Estad\'istica e Investigaci\'on Operativa, 
Universidad de Alicante,03080-Alicante España}
\email{casorranamilburu@msn.com}

\address{Severin Barmeier,
Graduate School of Mathematical Sciences,
The University of Tokyo,
3-8-1 Komaba, Meguro,
Tokyo, 153-8914 Japan}
\email{s.barmeier@googlemail.com}


\thanks{The second and third authors acknowledge support of the
Royal Society and of the Glasgow Mathematical Journal.}

\begin{abstract}  We give infinitely many new isomorphisms between
  moduli spaces of
bundles on local surfaces and on local Calabi--Yau threefolds.
\end{abstract}

\maketitle
\tableofcontents

%\markboth{}

\section{Introduction}\label{intro}

To study moduli spaces
of rank 2 bundles on local surfaces and local threefolds we
present concrete descriptions of these moduli
as quotients of the vector spaces of extensions of line bundles by holomorphic
isomorphism.
Our favourite varieties are the following:

 $$Z_k:=\Tot({\mathcal O}_{\mathbb P^1}(-k))\qquad  \text{and} \qquad
W_i:=\Tot ({\mathcal O}_{\mathbb P^1}(i-2)\oplus {\mathcal O}_{\mathbb P^1}(-i))
\text{,}$$
 together with moduli of bundles on them.
Let $\ell$ denote the zero section of $Z_k $
and denote by $X_k$ the surface obtained from $Z_k$
by contracting $\ell$ to a point; thus $X_k$ is singular for $k>1$.
For a bundle $E$ on a surface $Z_k$, let $\ell$ denote the zero section of $\mathcal{O}_{\mathbb{P}^1}(-k)$ considered as a subvariety of $Z_k$, and $\pi \colon Z_k \to X_k$ the map that contracts $\ell$ to a point $x$.  Hence $\pi$ is the inverse of blowing up $x$.  In the following, we shall also let $Y$ denote either $W_i$ or $Z_k$.

\begin{definition}
The {\it charge} of a bundle $E\to Y$ around $\ell$ is the {\it local holomorphic Euler characteristic} of $\pi_*E$ at $x$, defined as
\begin{equation}\label{eq.euler}
  \chi\bigl(x, \pi_*E\bigr) := \chi\bigl(\ell, E\bigr)
  := h^0\bigl(X; \; (\pi_*E)^{\vee\vee} \bigl/ \pi_* E \bigr)
  + \sum_{i=1}^{n-1}(-1)^{i-1} h^0\bigl(X; \; R^i \pi_* E\bigr) \text{ .}
\end{equation}
\end{definition}

\noindent Note that we have only $\chi\bigl(\ell, E\bigr) = h^0\bigl(X; \; (\pi_*E)^{\vee\vee}\bigl/ \pi_* E \bigr) + h^0\bigl(X; \; R^1 \pi_* E\bigr)$ since our spaces only have two coordinate charts (see \ref{can2}).

\begin{definition} \label{definicion} Let $\sim$ denote bundle isomorphism and introduce the following notation and definitions.
	\begin{enumerate}
		\item $ \mathcal M_{j_1,j_2}(Y) :=
 \Ext^1(\mathcal O_Y(j_2),\mathcal O_Y(j_1))\bigm/\sim $
		\item  ${\mathcal M}_j(Y, 0) := {\mathcal M}_{j, -j}(Y)$
		\item  ${\mathcal M}_j(Y, 1) := {\mathcal M}_{j+1, -j}(Y)$
	\end{enumerate}

\noindent Note that the second entry, that is either $1$ or $0$, denotes the {\it first Chern class} of the bundles considered in each case.

\noindent From such quotients we extract the following moduli spaces. Let
$\epsilon = 0 $ or $1$.
	
	\begin{enumerate}
		\item ${\mathfrak M}^1_j(Y, \epsilon) \subset {\mathcal M}_j(Y, \epsilon)$ consisting of elements given by an extension class vanishing to order exactly 1 over $\ell$,
		\item ${\mathfrak M}^s_j(Y, \epsilon) \subset {\mathfrak M}^1_j(Y, \epsilon)$ consisting of elements with lowest charge $\chi_{\rm low}$,
 where
 $\chi_{\rm low}:= \inf\{\chi(E) \vert E \in {\mathfrak M}^1_j(Y, \epsilon)\}$.
	\end{enumerate}
\end{definition}

\begin{remark}
For $W_1$, it follows by lemma \ref{alg} that all rank $2$ bundles are extensions of line bundles.  In fact, we also have this filtrability for $W_2$ but not for $W_i$ with $i\geq 3$.
\end{remark}

Our main results are the following:
\vspace{5mm}

\noindent{\bf Theorem }{\it  (Coincidence of moduli of bundles on surfaces and threefolds)}

For all positive integers $i,j,k$, there are isomorphisms
$$\mathfrak M^1_{2j+\left\lfloor\frac{k-3}{2}\right\rfloor+\delta}(Z_k,\epsilon) \simeq \mathfrak M^1_j(W_i, \delta)$$
and birational equivalences
$$\mathfrak M^s_{2j+\left\lfloor\frac{k-3}{2}\right\rfloor+\delta}(Z_k, \epsilon) \dashrightarrow
\mathfrak M^s_j(W_1, \delta)$$
when $\epsilon\equiv k+1\,{\rm mod}\,2$ and $\delta \in \{0,1\}$.
\vspace{5mm}
% The case $j=1$, $\delta=0$ gives moduli spaces $\mathfrak M_1^1(Z_k,\epsilon)$ and $\mathfrak M_1^1(W_i,0)$, which are points, giving the required isomorphism, which follows from \ref{poly} and \cite[prop.~3.22]{K}

\noindent{\bf Theorem }{\it  (Atiyah--Jones type statement for local moduli)}

For $q \leq 2(2j - k -2+\delta)$ there are isomorphisms
	\begin{itemize}
		\item[($\iota$)] $H_q(\mathfrak M^1_j(Z_k), \delta) = H_q(\mathfrak M^1_{j+1}(Z_k), \delta)$
		\item[($\iota\iota$)] $\pi_q(\mathfrak M^1_j(Z_k), \delta)= \pi_q(\mathfrak M^1_{j+1}(Z_k), \delta)\text{.}$
	\end{itemize}
	and for $q \leq 2(4j - 3-2\delta)$ there are isomorphisms
	\begin{itemize}
		\item[($\iota\iota\iota$)] $H_q(\mathfrak M^1_j(W_i), \delta) = H_q(\mathfrak M^1_{j+1}(W_i), \delta)$
		\item[($\iota\nu$)] $\pi_q(\mathfrak M^1_j(W_i), \delta)= \pi_q(\mathfrak M^1_{j+1}(W_i), \delta)\text{.}$
	\end{itemize}
\vspace{5mm}

\begin{remark}
We obtain isomorphisms between bundles $E$ and $F$ over $Z_k$ with $c_1(F)=c_1(E)+2$ by tensoring with $\mathcal O(-1)$, as
\[
\left(\begin{matrix}
z^{-j_1} & p \\ 0 & z^{-j_2}
\end{matrix}\right)\otimes z =
\left(\begin{matrix}
z^{-j_1+1} & zp \\ 0 & z^{-j_2+1}
\end{matrix}\right)
\]
so that we could consider $\epsilon\in\mathbb Z$, as long as $\epsilon\equiv k+1\,\mathrm{mod}\,2$ still holds.

\end{remark}

\section{Filtrability and algebraicity}\label{sec.alg}

We deal with  bundles on local surfaces and threefolds,
that is, a neighborhood of  a curve $C$ embedded
in a smooth surface or threefold $W$,
typically the total space of a vector bundle $N$ over $C$.
We  focus on the case when
$C \simeq \mathbb P^1$. In the $2$-dimensional case we
focus on the case when $N^*$ is ample, and in the $3$-dimensional
case we focus on  Calabi--Yau threefolds.

Let $W$ be a connected complex manifold (or smooth algebraic variety)
and $C$ a curve contained in $W$ that is reduced, connected and
a local complete intersection. Let $\widehat C$ denote the formal
completion of $C$ in $W$. Ampleness of the conormal bundle
 has a strong influence on the
behaviour of bundles on $\widehat C$.
 We will use the following basic fact from formal geometry.
\vspace{3mm}

\begin{lemma}\label{alg}\cite[thm.~3.2]{BGK2}
If the conormal bundle $N^*_C$ is ample, then every vector
bundle on $\widehat C$ is filtrable. If in addition $C$ is smooth,
then every holomorphic bundle on $\widehat C$  is algebraic.
\end{lemma}

\begin{remark} Ampleness of $N^*_C$ is essential. For example, consider
the Calabi--Yau threefold
$$W_i = \Tot ({\mathcal O}_{\mathbb P^1}(i-2)\oplus {\mathcal O}_{\mathbb P^1}(-i))
\text{.}$$
Then $W_1$ satisfies the hypothesis of \ref{alg}, hence holomorphic bundles
on $W_1$ are filtrable and algebraic, whereas on $W_2$ filtrability still
holds, but there exist proper
holomorphic bundles $W_2$ that are not algebraic, and on $W_i$ for $i\geq 3$ neither
filtrability nor algebraicity hold, see \cite{K} chapter 3.3.
\end{remark}

\section{Surfaces}\label{2}

 We use the very concrete description of moduli spaces
of rank 2 bundles over  the surfaces $Z_k :=
\Tot (\mathcal O_{{\mathbb P}^1}(-k))$ given in \cite{BGK1}.
Let $\ell$ denote the zero section inside $Z_k$. Given a bundle
$E$ over $Z_k$, its restriction to $\ell$ splits by Grothendieck's
principle, and if $E\vert_\ell \simeq {\mathcal O}(a_1) \oplus \cdots \oplus
{\mathcal  O}(a_r)$ then $(a_1, \dots, a_r)$ is called the {\it  splitting type}
 of $E$.
By \cite[thm.~3.3]{CA1}, a holomorphic
bundle $E$ over $Z_k$ having
splitting type $(j_1,j_2)$ with $j_1 \leq j_2$
can be written as  an algebraic extension

\begin{equation}\label{eq!ext}
  0 \longrightarrow \mathcal O(j_1) \longrightarrow E \longrightarrow \mathcal O(j_2) \longrightarrow 0
\end{equation}
 and therefore corresponds to an {\it extension class}
\[  p  \in \Ext^1_{Z_k}\!(\mathcal O(j_2),\; \mathcal
O(j_1)) .\]

We fix once and for all
coordinate charts on our surfaces $Z_k = U \cup V$, where
\begin{equation}\label{can2}
  U = \mathbb C^2_{z,u} = \{(z,u) \in \mathbb C^2\} \qquad\text{and}\qquad
  V = \mathbb C^2_{\xi,v} = \{(\xi, v) \in \mathbb C^2\} \end{equation}
and$$ (\xi, v) = (z^{-1}, z^ku)  \,\,\,\ \mbox{on}
\,\,\,\ U \cap V.$$
In these coordinates, the bundle $E$
may be represented by a transition matrix in   {\it canonical form} as
\[
T =
\left(\begin{matrix}{z^{-j_1}}& {p}\\ {0}& {z^{-j_2}}\end{matrix}
\right)
\]
where

\begin{equation}\label{poly}
  p =\hspace{-0.2cm}\sum_{i=1}^{\lfloor \vphantom{\textstyle X}(j_2 - j_1 - 2)/k \rfloor}\hspace{-0.25cm}
  \sum_{l = ki + j_1 + 1}^{j_2 - 1} p_{il} \, z^l u^i \text{ .}
\end{equation}
%% \begin{theorem}\cite[thm.~4.9]{BGK1}\label{thm.proj}
%% On the first infinitesimal neighbourhood, two bundles $E^{(1)}$ and
%% $E^{\prime(1)}$ with respective transition matrices
%% \[ \begin{pmatrix} z^j & p_1 \\ 0 & z^{-j} \end{pmatrix} \quad\text{and}\quad
%%    \begin{pmatrix} z^j & p'_1 \\ 0 & z^{-j} \end{pmatrix} \]
%% are isomorphic if and only if $p^\prime_1 = \lambda p_1$ for some
%% $\lambda \in \mathbb{C}^\times$.
%% \end{theorem}
Since we are interested in isomorphism classes of vector bundles
rather than extension classes, we use the following moduli:

\[  \mathcal M_{j_1,j_2}(Z_k) =
 \Ext^1(\mathcal O_{Z_k}(j_2),\mathcal O_{Z_k}(j_1))\bigm/\sim  \]
where $\sim $ denotes bundle isomorphism. We observe that this
quotient gives rise to a moduli stack, but we will only describe here
subsets of its coarse moduli space considered as a variety.
Considered just as a topological space, the full quotient
will not be Hausdorff except in the trivial case, when
it contains only a point. The latter happens when
%$\Ext^1\left(\mathcal O(j_2), \mathcal O(j_1)\right)=0.$
the only bundle with splitting type $(j_1,j_2)$ is
$ \mathcal O_{Z_k}  (j_1)\oplus\mathcal O_{Z_k}  (j_2)$, that is, whenever $j_2-j_1 <k+2$.

To specify the topology  in this quotient space, we use the canonical form of the
extension class  (\ref{poly}). Then the
coefficients of $p$ written in lexicographical order form
a vector in  $\mathbb C^m$, where $m$  is the number of complex coefficients
appearing in the expression of $p$.
%$N := \bigl\lfloor \frac{2j-2}k \bigr\rfloor$ as
We define an  equivalence relation in $\mathbb C^m$ by setting
$p \sim p'$ if  $(j_1,j_2,p)$ and $(j_1,j_2,p')$
define isomorphic bundles
over $Z_k$, and give $\mathbb C^m\bigl/\sim$ the quotient topology.
Now setting
$n := \lfloor (j_2-j_1-2)/k \rfloor$,
we obtain a bijection
\begin{eqnarray*}
  \phi \colon \mathcal M_{j_1,j_2}(Z_k) &\to& \mathbb C^m\bigl/\sim \text{ ,} \\
   \begin{pmatrix} z^{-j_1} & p \\ 0 & z^{-j_2}\end{pmatrix}
   &\mapsto& \bigl(p_{1,k+j_1+1}, \dotsc, p_{n,j_2-1}\bigr)
\end{eqnarray*}
and give $\mathcal M_{j_1,j_2}(Z_k)$
 the topology induced by this bijection.

Now observe that it is always the case that $p \sim \lambda p $ for
any $\lambda \in \mathbb C - \{0\}$. The
 moduli space is then evidently non-Hausdorff,
  as the only open  neighborhood of
the split bundle is the entire moduli space. In the spirit of GIT
one would like to extract nice moduli spaces out of these quotient
spaces. Clearly the split bundle needs to be removed, but there is
quite  a bit more topological complexity.

 \subsection{Vanishing $c_1$ case: moduli spaces}\label{2.0}
For rank 2 bundles $E$ over $Z_k$ with $c_1(E)=0$ there is a non-negative integer $j$
such that $E \vert_\ell \simeq {\mathcal O}(j) \oplus {\mathcal O}(-j)$
and we will say $E$ has splitting type $j$.
We denote by $\mathcal M_j$ the moduli  of all bundles with
this fixed splitting type (see Definition \ref{definicion}, item (2)):

\[ \mathcal M_j(Z_k,0) := \Ext^1(\mathcal O_{Z_k}(-j),\mathcal O_{Z_k}(j))\bigm/\sim  \text{.}\]

We now recall some results about the topological structure of
these spaces and their relation to instantons.
These moduli spaces are stratified into Hausdorff components
by local analytic invariants. Given a reflexive sheaf $E$ over $Z_k$ we set:
\[{\bf w}_k(E):= h^0((\pi_*E)^{\vee\vee}/\pi_*E),\qquad
{\bf h}_k(E):=h^0(R^1\pi_* E)\text{.}\]
called the {\it width} and {\it height} or $E$, respectively.
\begin{definition}
 $\chi(\ell,E) := {\bf w}_k(E)+{\bf h}_k(E)$ is called the
{\it local holomorphic
Euler characteristic} or {\it charge} of $E$.
\end{definition}

We quote the following results to show the connection with mathematical physics

\begin{theorem}\cite[cor.~5.5]{BGK1} {\it  Correspondence with
instantons.}
An $\mathfrak{sl}(2,\mathbb{C})$-bundle
over $Z_k$ represents an instanton if and only if its splitting type
is a multiple of $k.$
\end{theorem}

\begin{theorem}\cite[thm.~4.15]{BGK1} {\it  Stratifications.}
If $j=nk$ for some $n \in \mathbb{N}$, then the pair $({\bf h}_k,{\bf w}_k)$
stratifies instanton moduli stacks $\mathcal M_j(k)$ into Hausdorff
components.
\end{theorem}

\begin{remark}\label{usu} Let us note the following:
\begin{itemize}
\item $\chi$ alone is not fine enough to
stratify the moduli spaces.
\item Constructing such a stratification
for the non-instanton case is an open problem.
\item There are various ways to obtain moduli spaces inside the
$\mathcal M_j$. One possible choice is to take the largest Hausdorff
component as our moduli space. This will produce compact moduli, and
we study this case in section
\ref{order1}. A second, more natural choice is to
fix some numerical invariant,
to which end the local holomorphic Euler characteristic
presents itself as the most natural candidate.
\end{itemize}
\end{remark}

\subsection{Vanishing $c_1$ case: first order deformations}\label{order1}
\begin{notation}
Let ${\mathfrak M}^1_j(Z_k,0) \subset {\mathcal M}_j(Z_k)$ denote the
subset which parametrizes  isomorphism classes of
bundles on $Z_k$ consisting of isomorphism classes
of nontrivial
first order deformations of ${\mathcal O}(j) \oplus {\mathcal O}(-j)$,
that is, bundles $E$ fitting into an exact sequence
\begin{equation}\label{filt}
0 \rightarrow {\mathcal O}(-j) \rightarrow E \rightarrow
{\mathcal O}(j) \rightarrow 0 \end{equation}
 whose corresponding extension class
vanishes to order exactly one on $\ell$ (note that this excludes the split
bundle itself). In other words, 
 ${\mathcal I}_\ell=\langle u\rangle$ on the $u$-chart and consider only
extensions $p \in \Ext({\mathcal O}(j),
{\mathcal O}(-j))$  with  $p= u q$ and $u \nmid q$.
\end{notation}

\begin{remark} If $2j-2<k$ then $\mathcal M_j(Z_k) $ consists of
just a point represented by the split bundle, consequently
if $2j-2<k$ then ${\mathfrak M}^1_j(Z_k,0) = \emptyset.$
\end{remark}

A simple observation, which we now describe, then implies that  ${\mathfrak M}^1_j(Z_k)$ is
compact and smooth.
\vspace{3mm}

\begin{theorem}\label{inf}\cite[thm.~4.9]{BGK1}\label{thm.proj}
On the first infinitesimal neighbourhood, two bundles $E^{(1)}$ and
$F^{(1)}$ with respective transition matrices
\[ \begin{pmatrix} z^j & p_1 \\ 0 & z^{-j} \end{pmatrix} \quad\text{and}\quad
   \begin{pmatrix} z^j & q_1 \\ 0 & z^{-j} \end{pmatrix} \]
are isomorphic if and only if $q_1 = \lambda p_1$ for some
$\lambda \in \mathbb{C}-\{0\}$.
\end{theorem}

\begin{remark}
Note that no similar result holds true if we include  higher order
deformations, because then there are further identifications and
the quotient space is no longer Hausdorff.
\end{remark}

\begin{corollary}\label{cor}
 ${\mathfrak M}^1_j(Z_k,0) \simeq \mathbb P^{2j-k-2}\text{.}$
\end{corollary}

\subsection{Vanishing $c_1$ case: minimal charge}\label{chimin}

Another possible choice of moduli space, more compatible with the physics
motivation, is to
 consider the subset of bundles on ${\mathfrak M}^1_j(Z_k,0)$
 having fixed  charge; this is preferable, because the
charge is an analytic  invariant on the bundles, and minimal
charge corresponds to a generic choice for the corresponding
instanton interpretation. In this case we take the open subset
of the moduli of first order deformations  defined by:
  $${\mathfrak M}^s_j(Z_k,0):=
\{E   \in {\mathfrak M}^1_j(Z_k,0) :  \chi(E) = \chi_{\rm min}(Z_k)\}\text{.}$$
Charge is lower semi-continuous on the splitting type,
and we have that the locus of bundles with charge higher than $\chi_{\rm min}$
is Zariski closed; in fact, such locus
is determined by $k+1$ polynomial equations \cite[thm.~4.11]{BGK1}.

\begin{corollary}\label{quasi}
 ${\mathfrak M}^s_j(Z_k,0) $ is a quasi-projective
variety, whose complement in $\mathbb P^{2j-k-2}$ is cut out by
$k+1$ equations.

\begin{proof}
On the first infinitesimal neighbourhood $p_1$ has $2j-k-1$ coefficients modulo projectivisation (see equation \ref{poly}) and then, by means of Theorem \ref{inf}, we arrive at the desired result.
\end{proof}

\end{corollary}

\subsection{Case $c_1=1$}

From expression (\ref{poly}) we can read off the case $c_1=1$ by setting $j_1=-j$ and $j_2=j+1$,
considering extensions ${\rm Ext}^1_{Z_k}(\mathcal O(j+1),\mathcal O(-j))$.
The form of the extension class restricted to the first infinitesimal neighborhood expressed  in canonical coordinates is
$$ \sum_{l = k - j + 1}^{j} p_{1l} \, z^l u.$$
The coefficients vary in $\mathbb C^{2j-k}$, so that modulo the relation $p\sim \lambda p'$
we have:
\begin{lemma}\label{z1}
$\mathfrak M_j^1(Z_k,1)\simeq \mathbb P^{2j-k-1}$.
\end{lemma}
\begin{proof}
The proof of this lemma is just a modification of the proof of Theorem \ref{thm.proj}, which goes through successfully by replacing the appropriate $j$s with $j+1$: On the first infinitesimal neighbourhood, two bundles $E^{(1)}$ and $F^{(1)}$ with respective transition matrices
\[ \begin{pmatrix} z^j & p_1 \\ 0 & z^{-j-1} \end{pmatrix} \quad\text{and}\quad
   \begin{pmatrix} z^j & q_1 \\ 0 & z^{-j-1} \end{pmatrix} \]
are isomorphic if and only if $q_1 = \lambda p_1$ for some
$\lambda \in \mathbb{C}-\{0\}$.  Thus, projectivising the space of bundles on the first formal neighbourhood gives the isomorphism classes in the case $c_1=1$
just like we had in the vanishing $c_1$ case.
\end{proof}

The moduli space ${\mathfrak M}^s_j(Z_k,1)$ of bundles with minimal charge can also be considered as well. Since charge is lower semi-continuous, the set  ${\mathfrak M}^s_j(Z_k,1)$
of bundles in  ${\mathfrak M}^1_j(Z_k,1)$ achieving minimal charge is Zariski open.

\section{Threefolds}\label{3}

Consider the threefolds
$$W_i = \Tot ({\mathcal O}_{\mathbb P^1}(i-2)\oplus {\mathcal O}_{\mathbb P^1}(-i))
$$ to which we alluded earlier in section \ref{sec.alg}, and denote by $\ell$
the zero section inside $W_i$. We focus on the cases of
rank 2 and either $c_1=0$ or else $c_1=1$ as we did
in  section \ref{2} and for a bundle $E$ over $W_i$ such that
$E\vert_\ell \simeq {\mathcal O}(j) \oplus {\mathcal O}(-j) $
we call the non-negative integer $j$ the splitting type of $E$.
Note that here again $\Pic W_i \simeq \Pic \, \ell$ so we can avoid
a subscript in the notation ${\mathcal O}(j)$.

We now consider only {\it  algebraic extensions} over the $W_i$ and then
define moduli spaces analogous to the ones we defined in section
\ref{2}. First the set of isomorphism classes of bundles
with fixed splitting type:

\[ \mathcal M_j(W_i) =
 \bigl\{ E\to W_i : E\vert_{\ell} \simeq
\mathcal O (j)\oplus\mathcal O (-j) \bigr\}
 \bigm/ \sim \text{,}\]
and $${\mathfrak M}^1_j(Z_k) \subset {\mathcal M}_j(Z_k)$$  the
subset which parametrizes
bundles on $W_i$ which are nontrivial
 first order deformations of ${\mathcal O}(j) \oplus {\mathcal O}(-j)$,
that is, bundles $E$ fitting into an exact sequence
$$
0 \rightarrow {\mathcal O}(-j) \rightarrow E \rightarrow
{\mathcal O}(j) \rightarrow 0 $$
 whose corresponding extension class
vanishes to order exactly one on $\ell$ (note that this excludes the split
bundle itself).
In local canonical coordinate charts, we have
\begin{equation}\label{can3}
W_i = U \cup V, \qquad\text{with}\qquad U =\mathbb C^3= \{(z,u_1,u_2)\},
\qquad
V=\mathbb C^3 = \{(\xi,v_1,v_2)\}\qquad \end{equation}
$$ \text{and}\qquad (\xi,v_1,v_2) = (z^{-1},z^{2-i}u_1,z^{i}u_2)
\qquad\text{in}\qquad U \cap V\text{.}$$
Then on the $U$-chart
${\mathcal I}_\ell= \langle u_1,u_2\rangle$ and elements of ${\mathfrak M}^1_j$
are determined by extension classes
$p \in \Ext({\mathcal O}(j),
{\mathcal O}(-j))$  with either  $p= u_1p'$ or else $p=u_2p''$
and $u_1 \nmid p'p''$, $u_2 \nmid p'p''$.

\begin{lemma}\label{thm:gk}\cite[cor.~5.6]{GK} We have an isomorphism of varieties
$${\mathfrak M}^1_j(W_i,0) \simeq {\mathbb P}^{4j-5}\text{.}$$
\end{lemma}

Once again, fixing a numerical invariant seems to be a preferable
choice (as suggested by the last item on Remark \ref{usu}), so we define:
  $${\mathfrak M}^s_j(W_i,0):=
\{E   \in {\mathfrak M}^1_j(W_i,0) :  \chi(E) = \chi_{\rm min}(W_i)\}\text{,}$$
and this is a Zariski open subvariety of ${\mathfrak M}^1_j$.

\begin{lemma}\label{w1}
$\mathfrak M_j^1(W_i, 1) = \mathbb P ^{4j-3}$.
\begin{proof}
%	We can write $W_i=U\cap V$ where $U=\mathbb{C}^3$ and $V=\mathbb{C}^3$ with coordinates $\{(z,u_1,u_2)\}$ and $\{(\xi,v_1,v_2)\}$, respectively. We have the following chage of coordinates 	
%	 $$(\xi,v_1,v_2) \mapsto (z^{-1}, z^{2-i}u_1,z^iu_2)\text{.} $$
In canonical coordinates, an extension of $\mathcal{O}(j+1)$ by $\mathcal{O}(-j)$ may be represented over $W_i$ by the transition matrix:

\[
T =
\left(\begin{matrix}{z^{j}}& {p}\\ {0}& {z^{-j-1}}\end{matrix}
\right)\text{.}
\]
On the intersection $U\cap V = \mathbb C-\{0\}\times \mathbb C^2$ the holomorphic functions are of the  $$p=\sum_{t=-\infty}^{\infty}\sum_{s=0}^\infty\sum_{r=0}^\infty p_{rst} z^ru_1^s u_2^t\text{.}$$
By changing coordinates one can show that it is equivalent to consider $p$ as

\begin{align*}
&(p_{-j,0,0}z^{-j}+\dotsb + p_{j-1,0,0}z^{j-1}) \\
+ & (p_{-j-i+2,1,0}z^{-j-i+2}+\dotsb + p_{j-1,1,0}z^{j-1})u_1 \\
+ &(p_{-j+i,0,1}z^{-j+i} + \dotsb + p_{j-1,0,1}z^{j-1})u_2 \\
+ &\text{ higher-order terms}.
\end{align*}

Therefore, counting coefficients on the first infinitesimal neighbourhood gives $4j-2$ coefficients giving dimension $4j-3$ after projectivising.

\end{proof}
\end{lemma}

\begin{theorem}
For all positive integers $i,j,k$, there are isomorphisms
$$\mathfrak M^1_{2j+\left\lfloor\frac{k-3}{2}\right\rfloor+\delta}(Z_k,\epsilon) \simeq \mathfrak M^1_j(W_1, \delta)$$
and birational equivalences
$$\mathfrak M^s_{2j+\left\lfloor\frac{k-3}{2}\right\rfloor+\delta}(Z_k, \epsilon) \dashrightarrow
\mathfrak M^s_j(W_1, \delta)$$
when $\epsilon\equiv k+1\,{\rm mod}\,2$ and $\delta \in \{0,1\}$.
\end{theorem}

\begin{proof}
By setting $j\mapsto 2j+\left\lfloor\frac{k-3}{2}\right\rfloor+\delta$ in Corollary \ref{cor}, we obtain isomorphisms \[
\mathfrak M^1_{2j+\left\lfloor\frac{k-3}{2}\right\rfloor+\delta}(Z_k,0) \simeq \mathbb P^{4j-3-2\delta}
\]
for $k$ odd.  Similarly, we can use lemma \ref{z1} to obtain isomorphisms
\[
\mathfrak M^1_{2j+\left\lfloor\frac{k-3}{2}\right\rfloor+\delta}(Z_k,1) \simeq \mathbb P^{4j-3-2\delta}
\]
for $k$ even. The required isomorphisms to $\mathfrak M^1_j(W_1, \delta)$ then follow from lemmas \ref{thm:gk} and \ref{w1} for $\delta=0,1$, respectively.

To find the birational equivalences, first note that we have $$\mathfrak M^s_{2j+\left\lfloor\frac{k-3}{2}\right\rfloor+\delta}(Z_k, \epsilon) \subset \mathfrak M^1_{2j+\left\lfloor\frac{k-3}{2}\right\rfloor+\delta}(Z_k, \epsilon) \text{ and } \mathfrak M^s_j(W_i, \delta)\subset \mathfrak M^1_j(W_i, \delta)$$by definition.  Lemma \ref{quasi} shows that $\mathfrak M^s_{2j+\left\lfloor\frac{k-3}{2}\right\rfloor+\delta}(Z_k, \epsilon)$ is a quasi-projective variety and we now show that $\mathfrak M^s_j(W_1, \delta)$ is also quasi-projective.

For any bundle on $W_1$, \cite[lem.~5.2]{BGK2} shows that the width is always ${\bf w}(E) = h^0\left( (\pi_*E)^{\vee\vee}\bigm/ \pi_*E\right) = 0$.  Thus, fixed charge is equivalent to fixed height.  Since height is minimal on a Zariski open set of $W_1$ of codimension at least $3$ given by the vanishing of certain coefficients of $p$, ${\mathfrak M}^s_j(W_1)$ is Zariski open in ${\mathfrak M}^1_j(W_1)$.

Restricting the isomorphisms above to a suitably small neighbourhood of these quasi-projective varieties then gives the required birational equivalences.
\end{proof}

\begin{question}
Since $\ell\subset W_i$ cannot be contracted to a point for $i>1$, our definition of charge does not apply.  Can similar numerical invariants be defined for bundles on $W_i$, $i>1$? Some such invariants were defined in \cite{K} chapter 3.5, though much remains to be understood about their geometrical meaning.
\end{question}

\begin{theorem}
For $q \leq 2(2j - k -2+\delta)$ there are isomorphisms
	\begin{itemize}
		\item[($\iota$)] $H_q(\mathfrak M^1_j(Z_k), \delta) = H_q(\mathfrak M^1_{j+1}(Z_k), \delta)$
		\item[($\iota\iota$)] $\pi_q(\mathfrak M^1_j(Z_k), \delta)= \pi_q(\mathfrak M^1_{j+1}(Z_k), \delta)\text{.}$
	\end{itemize}
	and for $q \leq 2(4j - 3-2\delta)$ there are isomorphisms
	\begin{itemize}
		\item[($\iota\iota\iota$)] $H_q(\mathfrak M^1_j(W_i), \delta) = H_q(\mathfrak M^1_{j+1}(W_i), \delta)$
		\item[($\iota\nu$)] $\pi_q(\mathfrak M^1_j(W_i), \delta)= \pi_q(\mathfrak M^1_{j+1}(W_i), \delta)\text{.}$
	\end{itemize}
\end{theorem}

\begin{proof}
The statements follow immediately from corollary \ref{cor} and lemmas \ref{z1}, \ref{thm:gk}  and \ref{w1}.
\end{proof}
%\begin{remark}
%	We use the following estimates of codimension:
%	$$
%	d_k^2 = dim\left(\mathfrak M^1_j(Z_k) - \mathfrak M^s_j(Z_k) \right) \geq k+1
%	\text{ and }
%	d_k^3 = dim\left(\mathfrak M^1_j(W_1) - \mathfrak M^s_j(W_1) \right) \geq 3
%	$$
%\end{remark}

\end{document}